\documentclass[amsfonts,12pt]{amsart}
\usepackage{color}

\newtheorem{thm}{Theorem}
\newtheorem*{thm*}{Theorem}
\newtheorem{cor}[thm]{Corollary}
\newtheorem{lem}[thm]{Lemma}
\newtheorem{prop}[thm]{Proposition}

\newtheorem{con}[thm]{Conjecture}
\newtheorem{ex}[thm]{Example}
\newtheorem{rem}[thm]{Remark}

\newtheorem{conj}[thm]{Conjecture}
\newtheorem{probl}[thm]{Problem}

\newcommand{\CF}{\mathcal{F}}

\newcommand{\CH}{\mathcal{H}}

\newcommand{\CK}{\mathcal{K}}

\newcommand{\CR}{\mathcal{R}}

\newcommand{\CV}{\mathcal{V}}
\newcommand{\CW}{\mathcal{W}}

\newcommand{\DS}{\displaystyle}

\newcommand{\R}{\mathbb{R}}
\newcommand{\B}{\mathbb{B}}
\renewcommand{\S}{\mathbb{S}}
\newcommand{\NP}{\mathbb{NP}}
\renewcommand{\P}{\mathbb{P}}
\newcommand{\Q}{\mathbb{Q}}
\newcommand{\Z}{\mathbb{Z}}

\newcommand{\PiK}{\Pi\mathrm{K}}
\newcommand{\PicK}{\Pi^{\circ}\mathrm{K}}
\newcommand{\PiP}{\Pi\mathrm{P}}
\newcommand{\PiQ}{\Pi\mathrm{Q}}

\newcommand{\DP}{\mathrm{DP}}
\newcommand{\LW}{\mathrm{LW}}

\newcommand{\inte}{\operatorname{int}}
\newcommand{\lin}{\operatorname{lin}}

\newcommand{\N}{\mathbb{N}}

\newcommand{\prob}{\operatorname{prob}}
\newcommand{\vol}{\operatorname{vol}}
\newcommand{\conv}{\operatorname{conv}}

\newcommand{\size}{\operatorname{size}}

\newcommand{\diam}{\operatorname{diam}}

\newcommand{\area}{\operatorname{S}}
\newcommand{\iso}{\operatorname{iso}}

\newcommand{\1}{{\hbox{\rm 1} \kern -.41em \hbox{\rm 1} }}

\begin{document}

\title{On the reverse Loomis--Whitney inequality}

\author[S. Campi]{Stefano Campi}
\address[]{Stefano Campi: Dipartimento di Ingegneria dell'Informazione, Universit\'a degli Studi di Sienna, I-53100 Siena, Italy}
\email{campi@dii.unisi.it}

\author[P. Gritzmann]{Peter Gritzmann}
\address[]{Peter Gritzmann: Zentrum Mathematik, Technische Universit\"at M\"unchen, D-85747 Garching bei M\"unchen, Germany}
\email{gritzmann@tum.de}

\author[P. Gronchi]{Paolo Gronchi}
\address[]{Paolo Gronchi: Dipartimento di Matematica e Informatica ``Ulisse Dini'', Universit\'a degli Studi di Firenze,
I-50122 Firenze, Italy}
\email{paolo.gronchi@unifi.it}

\begin{abstract}
The present paper deals with the problem of computing (or at least estimating) the $\LW$-number $\lambda(n)$, i.e.,  the supremum of all $\gamma$ such that for each convex body $K$ in $\R^n$ there exists an
orthonormal basis $\{u_1,\ldots,u_n\}$ such that
$$
\vol_n(K)^{n-1} \geq \gamma \prod_{i=1}^n \vol_{n-1} (K|u_i^{\perp}) ,
$$
where $K|u_i^{\perp}$ denotes the orthogonal projection of $K$ onto the hyperplane $u_i^{\perp}$ perpendicular to
$u_i$. Any such inequality can be regarded as a reverse to the well-known classical Loomis--Whitney inequality.
We present various results on such {\em reverse Loomis--Whitney inequalities}.
In particular, we prove some structural results, give bounds on $\lambda(n)$ and deal with the
problem of actually computing the $\LW$-constant of a rational polytope.
\end{abstract}
\maketitle

\section{Introduction}\label{sec-LW}

\subsection{The Loomis--Whitney inequality}
The  well-known {\em Loomis--Whitney inequality} \cite{lw-49} states that, for any Lebesgue measurable set $A$ in $\R^n$,
\begin{equation}\label{LoomisWhitneyineq}
\mu_n(A)^{n-1} \le \prod_{i=1}^n \mu_{n-1} (A|e_i^{\perp}),
\end{equation}
where $e_i$ is the $i$th standard coordinate vector of $\R^n$, $i\in [n]=\{1,\ldots,n\}$, $\mu_k$ denotes
the $k$-dimensional Lebesgue measure, and $A|u^{\perp}$ stands for the orthogonal projection of $A$
on the hyperplane $u^{\perp}$ through the origin perpendicular to $u\in \R^n\setminus \{0\}$.

The original proof by Loomis and Whitney is based on a combinatorial argument.
Note that, as coordinate boxes are examples of sets for which equality holds, (\ref{LoomisWhitneyineq}) is tight.
Recently, \cite{BBFL} and \cite{EFKY} gave general stability results which, in particular, yield characterizations
of equality in (\ref{LoomisWhitneyineq}). For instance, in the class of {\em bodies}, i.e., connected compact sets which are
the closure of their interior, (\ref{LoomisWhitneyineq}) holds with equality for, and only for, coordinate boxes;
see \cite[Corollary 2]{EFKY}).

Over the years, the Loomis--Whitney inequality has been extended, generalized and used
in other mathematical areas in various ways.
For instance, \cite{BollobasThomason} and \cite{BuragoZalgaller}  consider versions involving projections
on coordinate subspaces of any dimension, \cite{BalisterBollobas} establishes a connection with
Shearer's entropy inequality,
\cite{ball-91a} gives a generalization for projections along other suitable directions,
\cite{BCW}, \cite{BobkovNazarov},   \cite{BrascampLieb} prove extensions to analytical
inequalities, \cite{BCT} shows its relevance for the multilinear Kakeya conjecture,
and \cite{Gromov} studies its implications in group theory. The very recent papers \cite{BDGK-16},
\cite{BGL-16},  \cite{GKV-16} give various generalizations particularly
in the context of the uniform cover inequality, for the average volume of sections or
between quermassintegrals of convex bodies and extremal or average projections on
lower dimensional spaces, respectively.

Fields of applications of Loomis--Whitney type inequalities range from microscopy
and stereology \cite{Vedel}, to geochemistry \cite{Shepherd}, astrophysics
\cite{OstroConnelly} and information theory \cite{Donoho}, \cite{HWB}; see \cite{CampiGarGro} and \cite{CampiGro} for other examples.

\subsection{Reversing the inequality} While the Loomis--Whitney inequality gives
an estimate of the Lebesgue measure of a  set $A$ from above, the present paper asks for
a corresponding estimate from below.  Since the case of $n=1$ is trivial, we assume in the following
that $n\geq 2$.

Of course, in general $\mu_n(A)$ can be made arbitrarily small by removing suitable
subsets of interior points from $A$ without changing its projections.
Therefore, we are quite naturally led to the class $\CK^n$ of {\em convex bodies} (and we will
write $ \vol_k$ now rather than $\mu_k$).
This restriction already allows the reverse inequality by Meyer \cite{Meyer88}
\begin{equation}\label{Meyerinequality}
        \vol_n(K)^{n-1} \geq \frac{(n-1)!}{n^{n-1}} \prod_{i=1}^n \vol_{n-1} (K\cap e_i^{\perp})
\end{equation}
in terms of sections with coordinate hyperplanes instead of projections.
In (\ref{Meyerinequality})  equality holds if and only if $K$ is a {\em coordinate crosspolytope},
i.e., the convex hull of segments
$[-\alpha_ie_i, \alpha_ie_i]$ with $\alpha_i\in (0,\infty)$, $i\in [n]$.
However, even for convex bodies, only trivial reverse Loomis--Whitney inequalities exist if
we insist on projections onto coordinate hyperplanes.
In fact, let  $T_{n}$ denote the  $n$-dimensional regular simplex in $\R^n$. (We speak of {\em the} regular simplex
since our problem is invariant under similarity anyway). Further let  $\hat{T}_{n-1}= \conv(\{e_1,\ldots,e_n\})\subset \R^{n}$ be the standard embedding of  $T_{n-1}$ in $\R^{n}$.
Then $\vol_{n}(\hat{T}_{n-1})=0$ while $\vol_{n-1} (\hat{T}_{n-1}| e_i^{\perp})=1/(n-1)!$.
Hence, the family $\{S_\tau: \tau \in (0,\frac{1}{n})\}$ of the convex hulls of $\hat{T}_{n-1}$ and the
point $(\tau, \tau, \dots, \tau)\in \R^n$ cannot satisfy the inequality
$$
 \vol_n(S_\tau)^{n-1} \geq \gamma \prod_{i=1}^n \vol_{n-1} (S_\tau|e_i^{\perp})
$$
for any positive constant $\gamma$.
Therefore, we allow rotations first, i.e., we project on hyperplanes $u_i^\perp$ perpendicular to the vectors
$u_1,\ldots,u_n$ of an orthonormal basis of $\R^n$.

\subsection{Reverse Loomis--Whitney inequalities and $\LW$-numbers}
Any orthonormal basis $F=\{u_1,\ldots,u_n\}$ of $\R^n$ will be called
a {\em frame}, and the set of all frames $F$ will be denoted by $\CF^n$. Further, for $K\in \CK^n$ and
$F=\{u_1,\ldots,u_n\}\in \CF^n$ let
$$
\Psi (K;F)= \prod_{i=1}^n \vol_{n-1} (K|u_i^{\perp}), \qquad
\Lambda (K;F)= \frac{\vol_n(K)^{n-1}}{\Psi (K;F)}.
$$
The quantities $\Psi (K;F)$ and $\Lambda (K;F)$ will be called the {\em projection average} and the
{\em $\LW$-ratio} of $K$ for $F$, respectively. Then
$$
\Psi (K)= \min_{F\in \CF^n}\Psi (K;F), \qquad \Lambda (K)= \max_{F\in \CF^n}\Lambda (K;F)
$$
are the {\em minimal projection average} and the  {\em $\LW$-constant of $K$}, respectively.
Note that, by compactness of the Euclidean unit sphere in $\R^n$,
the minimum on the left and, hence, the maximum on the right are indeed attained.
A frame producing the minimum projection average $\Psi(K)$ for the given convex body $K$
will be called a {\em best frame} for $K$. Obviously, the $\LW$-ratio is invariant under
similarity. 

We are interested in the best {\em universal} $\LW$-constant, i.e., the {\em $\LW$-number} of $\CK^n$,
$$
\lambda(n)= \inf_{K\in \CK^n} \Lambda (K).
$$
In other words, we want to compute (or at least estimate)
the supremum $\lambda(n)$ of all $\gamma$ such that for each convex body $K$ in $\R^n$
there exists an orthonormal basis $\{u_1,\ldots,u_n\}$ satisfying
$$
\vol_n(K)^{n-1} \geq \gamma \prod_{i=1}^n \vol_{n-1} (K|u_i^{\perp}) .
$$
Any inequality of this type will be called a {\em reverse Loomis--Whitney inequality}.

The paper is organized as follows. Section \ref{sec-prelim} is devoted to some preliminaries.
Section \ref{sec-main} will give an overview of our main results and open problems,
Section \ref{sec-face-structure} will link a best frame for a polytope $P$ to a rectangular box of smallest volume containing the projection body of $P$ and will then show how best frames relate to the facets of $P$.
In Section \ref{sec-evaluating} we will estimate $\lambda(n)$ as well as $\Lambda (K)$ for special bodies like simplices and regular crosspolytope. Section \ref{sec-algorithms} will deal with
algorithmic issues. The final Section \ref{sec-functionals} will provide results
for some related functionals.

\section{Preliminaries}\label{sec-prelim}

For easier reference we collect in this section some notation used throughout the paper.
Specific terms related to computational aspects will be given at the beginning of
Section \ref{sec-algorithms}.

We shall denote by $\|\cdot\|$ the Euclidean norm, by $\B^{n}$ the unit ball and by $\S^{n-1}$ the unit sphere in $\R^n$. We shall mean by {\sl convex body} in $\R^n$ a compact convex subset of $\R^n$ with nonempty interior, and we shall denote by $\CK^n$ the class of $n$-dimensional convex bodies.
Moreover, $\CK^n_o$ will denote the subclass of all origin symmetric convex bodies.

In addition to the volume $\vol_n(K)$ we shall also need the {\em surface area} $\area(K)$ of
a convex body $K$ and the corresponding {\em isoperimetric ratio}
$$
\iso(K)= \frac{\vol_n(K)^{n-1}}{\area(K)^n}.
$$

For $K\in \CK^n$, the {\sl support function} $h_K$ is defined by
$$
   h_K(u)=\max \{u^Tx:x\in K\},\ \ u\in \S^{n-1},
$$
where $u^T$ stands for the transpose of $u$.

The {\sl radial function} $\rho_K$ of $K$ is
$$
    \rho_K(u)=\max \{ \lambda : \lambda u \in K\}, \ \ u\in \S^{n-1}.
$$

For any $K\in \CK^n$ containing the origin in its interior, the {\sl polar} of $K$ is the convex body $K^\circ$ such that
\begin{equation}\label{polar}
   \rho_K(u)h_{K^\circ}(u) = 1, \ \ u\in \S^{n-1}.
\end{equation}

The {\em projection body} $\PiK$ of $K$ is the convex body defined by
\begin{equation}\label{projbody}
h_{\PiK}(u) = \vol_{n-1}(K|u^\perp), \qquad u\in \S^{n-1}.
\end{equation}
The body $\PiK$ is origin symmetric; in fact, it is a full-dimensional {\em zonoid} (see e.g. \cite[p. 302]{Schneider}, \cite{SchneiderWeil}). If $P$ is a polytope in $\CK^n$, then $\PiP$ is a {\em zonotope}, i.e.,
the Minkowski sum of finitely many line segments. More precisely, if $F_1,\ldots,F_m$
denote the facets of $P$ with corresponding outward unit normals $a_1,\ldots,a_m$, then
\begin{equation}\label{eq-zonotope}
\PiP = \frac{1}{2} \sum_{j=1}^m \vol_{n-1}(F_j) [-a_j,a_j],
\end{equation}
where $[-a_j,a_j]$ denotes the line segment from $-a_j$ to $a_j$.

We shall make use of an inequality due to Zhang \cite{Zhang} involving the polar $(\PiK)^{\circ}$ of the
projection body $(\PiK)^{\circ}$ of $K$. Using the abbreviation $(\PicK)$ for $(\PiK)^{\circ}$, it says that, for $K\in \CK^n$,
\begin{equation}\label{prop-zang}
     \vol_n(K)^{n-1}\geq \frac{\binom{2n}{n}}{n^n\vol_n(\PicK)}
\end{equation}
with equality if and only if $K$ is a simplex.

Finally, a {\it rectangular crosspolytope} is the convex hull of segments $[-\alpha_iu_i,\alpha_iu_i]$ with $\alpha_i \in (0,\infty)$, $i\in [n]$, where $\{u_1,\dots, u_n\}$ is a frame. A {\it coordinate crosspolytope} is one corresponding to the standard orthonormal basis.

\section{Main results and open problems}\label{sec-main}

Our first theorem links best frames for polytopes to their
facets. More specifically, we have the following result.

\begin{thm}\label{thm-facets}
Let $P\in \CK^n$ be a polytope and let $F=\{u_1,\ldots,u_n\}$ be a best frame for $F$.
Then at least $n-1$ of the vectors $u_i$ are parallel to a facet of $P$.
\end{thm}

The proof of Theorem \ref{thm-facets} will be given in Section \ref{sec-face-structure}.
Our next result settles the problem in the plane.

\begin{thm}\label{Planarcase}
If $K$ is a planar convex body, then
$$\Lambda(K)\geq \frac{1}{2},$$
and equality holds if and only if $K$ is a triangle.
\end{thm}

The planar case suggests computing the $\LW$-constant $\Lambda(S_n)$ of $n$-dimensional simplices $S_n$.

\begin{thm}\label{thm-simplex}
Let $S_n$ be an $n$-dimensional simplex in $\R^n$, $n\geq 3$. Then
$$
\frac{(n-1)!}{2^{n-2}n^{n-1}}<\Lambda (S_n) \leq \Lambda (T_n) = \frac{(n-1)!}{n^{n-1}}.
$$
\end{thm}

The next result follows immediately from (\ref{Meyerinequality}) with the aid of
Theorem \ref{thm-facets} for excluding equality.

\begin{prop}\label{rem-special-bodies}
Let $K\in \CK^n$. If
$$
    \vol_{n-1}(K|e_i^\perp)=\vol_{n-1}(K\cap e_i^\perp),\qquad i\in [n]
$$
then
$$
\Lambda (K) > \frac{(n-1)!}{n^{n-1}}.
$$
\end{prop}

For the general case we have a somewhat weaker bound.

\begin{thm}\label{thm-lower bound}
For every $K\in \CK^n$,
$$
\Lambda (K) \geq \frac{\binom{2n}{n}}{(2n)^{n}}.
$$
\end{thm}

In a previous arXiv-version of the present paper we asked for the asymptotic behavior
of $\Lambda(n)$ as $n\rightarrow \infty$. Subsequently, the authors of \cite{KSZ-16} showed
that the correct order is
$$n^{-n/2},$$
which, in particular, implies that simplices do not provide the worst case.
Their proof relies heavily on previous work in \cite{ball-91a} and \cite{ball-91b}.

Theorems \ref{Planarcase}, \ref{thm-simplex} and \ref{thm-lower bound} will be proved
in Section \ref{sec-evaluating}.
In Section \ref{sec-algorithms} we will deal with various algorithmic issues related to
reverse Loomis--Whitney inequalities and give a tractability result in fixed dimension.
More precisely, we show that there is a pseudopolynomial algorithm for approximating
$\Lambda(P)$ whose running time depends on the binary size of the input, the
error and a lower bound on the isoperimetric ratio $\iso(P)$ of $P$.

\begin{thm}\label{thm-poly-time-fixed-dim}
Let the dimension $n$ be fixed.
For a given rational polytope $P$, a given $\nu \in \Q$ with $0<\nu \le \iso(P)$,
and a given bound $\delta \in (0,1) \cap \Q$ we can compute a number $\Lambda$ such that
$$
\Lambda \le \Lambda(P) \le (1+\delta) \Lambda
$$
in time that is polynomial in $\size(P)$, $1/\delta$ and $1/\nu$.
\end{thm}

Of course, the dependence on $\iso(P)$ can be dropped in Theorem \ref{thm-poly-time-fixed-dim}
if we restrict the polytopes to any class for which there exists a function $f(n,m)$ depending
only on $n$ and, for each fixed $n$, only polynomially on the number $m$ of facets, such that
$\iso (P)\ge f(n,m)$. We do not know whether there is an algorithm that only depends on
$\size(P)$ and $1/\delta$. In any case, we presume that, in Theorem \ref{thm-poly-time-fixed-dim},
the condition that $n$ is fixed cannot be lifted. Actually, we conjecture that the problem
\begin{quote}
Given a rational polytope $P$ (presented by its vertices or facet halfspaces) and a precision $p$,
approximate $\Lambda(P)$ up to $p$ binary digits
\end{quote}
is $\#\P$-hard; see Conjecture \ref{conj-hardness}.

In the final Section \ref{sec-functionals} we present various results on functionals
related to $\Lambda$.
In particular, we consider the functional $\Phi$ which associates with $K\in \CK^n$ the
volume ratio of $K$ and a rectangular box of minimal volume containing $K$.
The problem of finding a rectangular box of minimal volume enclosing
a body has been studied extensively for $n=2,3$ in computational geometry, as it appears
in various practical tasks. In particular, applications are reported in
robot grasping \cite{HRK}, industrial packing and optimal design \cite{CT},
hierarchical partitioning of sets of points or set theoretic estimation \cite{BHP}, \cite{Com},
and state estimation of dynamic systems \cite{MV}.

\section{Projection bodies and best frames for polytopes}\label{sec-face-structure}

We begin by observing that the best frames for a given $K\in \CK^n$ are related to
volume minimal rectangular boxes containing the projection body $\PiK$ of $K$.

\begin{lem}\label{rem-frame-box}
Let $K\in \CK^n$, let $B$ denote a rectangular box of smallest volume containing $\PiK$,
and let $\pm u_1, \ldots, \pm u_n$ be its outward unit normals.
Then  $F=\{u_1, \ldots, u_n\}$ constitutes a best frame for $K$, and vice versa.
\end{lem}

\begin{proof}
Let $F=\{ v_1, \ldots, v_n\}$ be a frame and let $B$ be the rectangular box of minimal volume containing $\PiK$ whose outward unit normals are $\pm v_1, \ldots, \pm v_n$. Then each facet of $B$ contains a point of $\PiK$. Moreover, $\PiK$ is origin symmetric and consequently $B$ is origin symmetric, too.
Hence
$$
\vol_n(B)= \prod_{i=1}^n \bigl(2h_B(v_i)\bigr)
= 2^n\prod_{i=1}^n h_{\PiK}(v_i) = 2^n\prod_{i=1}^n \vol_{n-1}(K|v_i^\perp)=2^n\Psi (K;F).
$$
Thus any smallest box containing $\PiK$ induces a best frame for $K$ and vice versa.
\end{proof}

As Lemma \ref{rem-frame-box} shows, the problem of determining $\Lambda(K)$ is indeed
a special instance of a {\em containment problem}; see \cite{gritzmann-klee-94a}
for a survey on general containment problems.
Actually, by Lemma \ref{rem-frame-box}, $\Lambda (K)$ is closely related to the functional
$$
       \Phi(K)=\max_{F\in \CF^n}\frac{\vol_n(K)}{\vol_n(B(K;F))},
$$
where $B(K;F)$ is the rectangular box of minimal volume containing $K$ with edges
parallel to the vectors in $F$.

Note that, in the plane, $\Lambda(K)$ is the ratio of $\vol_2(K)$ and the area of the minimal
area rectangle circumscribed about $K$, i.e., for $K\in \CK^2$ we have
$\Lambda (K) = \Phi (K)$. In general, the two functionals are related
through the {\em Petty functional}
\cite{Petty71}
$$
\Theta(K)=\frac{\vol_n(K)^{n-1}}{\vol_n(\PiK)};
$$
more precisely,
$$
   \Lambda (K) = 2^n\Phi (\PiK) \, \Theta(K).
$$

While the functional $\Theta$ has attracted much interest in the theory of convex bodies,
the determination of its extreme values is still open (see  \cite[pp. 375, 380]{Gardner} and \cite[pp. 577--578]{Schneider}). 
In particular, it was conjectured by Petty that
$\Theta$ attains its maximum if and only if $K$ is an ellipsoid. On the other hand, Brannen
\cite{brannen-96} conjectured that the minimum is attained if $K$ is a simplex.

As $\Lambda (K) = \Phi (K)$ for $K\in \CK^2$, if $K$ is a polygon, the following result reduces the computation of $\Lambda(K)$ to a comparison of finitely many values.

\begin{prop}[Freeman and Shapira \cite{FreemanShapira75}]\label{FreemanShapira}
Let $P$ be a planar convex polygon. Then every rectangle of minimal area containing
$P$ has at least one edge parallel to an edge of $P$.
\end{prop}

We now turn to dimensions $n\ge 3$. Even though results on minimal volume
enclosing rectangular boxes cannot be evoked directly anymore, we will see that
Lemma \ref{rem-frame-box} is still useful.
We first point out that in dimension three an analogous result to Proposition \ref{FreemanShapira} holds.

\begin{prop}[O'Rourke \cite{O'Rourke85}]\label{O'Rourke}
Let $P$ be a polytope in $\R^3$. Then every minimal volume rectangular box  containing $P$
has at least two adjacent faces both containing more than one point of $P$.
\end{prop}

Next, we prove a structural result in arbitrary dimensions.

\begin{thm}\label{GeneralizedO'Rourke}
Let $P$ be a convex polytope in $\R^n$, and let $B$ be a rectangular box of minimal volume
containing $P$.
Then $B$ has at most one pair of opposite facets whose intersections with $P$
are just a single point each.
\end{thm}

\begin{proof}
Let $B$ be a rectangular box of minimal volume containing $P$, and let $\pm v_1,\dots,\pm v_n$
be the outward unit normals of the facets of $B$.
Suppose that there are two pairs of opposite facets containing just one point of $P$ each,
and let $\pm v_1$ and $\pm v_2$ be their normals. Among all frames in $\R^n$, we focus
on those which contain $v_3,\ldots,v_n$, and have the remaining
two in the plane $V=\lin \{v_1,v_2\}$ spanned by $v_1$ and $v_2$.
All such frames are related to each other through a rotation in $V$, i.e., a rotation with
axis $W=\lin \{v_3,\ldots,v_n\}$.

In order to minimize the volume of the bounding box associated to each such frame, we
have to minimize the product of the widths of $P$ in two orthogonal directions of $V$.
Therefore, we need to minimize the area of the rectangle bounding $P|V$.
By Proposition \ref{FreemanShapira}, the minimum is attained by a rectangle which contains
an edge of $P|V$ in its boundary. So $B$ cannot have two pairs of opposite facets
each containing a unique point of $P$.
\end{proof}

We can now prove Theorem \ref{thm-facets}.

\begin{proof}[Proof of Theorem \ref{thm-facets}]
By Lemma \ref{rem-frame-box}, every best frame of $P$ corresponds to a minimal volume rectangular box $B$ containing $\PiP$.
Since $\PiP$ is origin symmetric, Theorem \ref{GeneralizedO'Rourke} guarantees that
each facet of $B$, except possibly one pair of opposite facets, contains at least two points
of $\PiP$ and consequently contains a face of $\PiP$ of dimension at least one.
Since $\PiP$ is a zonotope, each of its
faces is the Minkowski sum of some of its generating segments. In particular,
this means that each facet of $B$, except possibly one pair of opposite facets,
is parallel to at least one segment generating $\PiP$. Therefore each vector
of a best frame, with possibly one exception, is parallel to at least one facet of $P$.
\end{proof}

\section{Bounds on $\Lambda$}\label{sec-evaluating}

The computation of $\lambda(2)$ is based on the results of
Section \ref{sec-face-structure}, more precisely on Proposition \ref{FreemanShapira} and
Lemma \ref{rem-frame-box}.

\begin{proof}[Proof of Theorem \ref{Planarcase}]
In order to prove the inequality, it is enough to show that, for every planar convex body $K$,
there exists a frame $\{u_1, u_2\}$ such that the area of the
rectangle circumscribed about $K$ with edges parallel to $u_1$ and $u_2$ is not larger
than twice the area of $K$.

Let $a,b\in K$ be such that $\|a-b\|$ is the diameter $\diam K$ of $K$.
Further, let $u_1=(a-b)/\|a-b\|$ and let $u_2\in \S^1$ be perpendicular to $u_1$.
Then the lines $a+\lin\{u_2\}$ and $b+\lin\{u_2\}$ support $K$.
Let $c,d \in K$ be such that $c+\lin\{u_1\}$ and $d+\lin\{u_1\}$ are different supporting lines
to $K$. Then the four corresponding halfplanes define a rectangle $B$ circumscribed about $K$.
(Note that the four points $a$, $b$, $c$, $d$ need not all be distinct.)
Now set $P=\conv\{a, b, c, d\}$. Then we have
$$
\vol_2(B) = 2\vol_2(P)\leq 2\vol_2(K),
$$
with equality if and only if $K=P$.

Finally suppose that $\Lambda(K)=\frac12$. Then $K=P$ and $B$ is a rectangle of minimal
area containing $K$. Hence, by Proposition \ref{FreemanShapira}, one of the edges
of $B$ must contain two points of $\{a, b, c, d\}$. Since $a$ and $b$ are the
endpoints of a diameter of $K$, it follows that $c$ or $d$ belongs to the line segment $[a,b]$. Hence $K$ must be a triangle.
\end{proof}

The result $\lambda(2)=\frac12$ of Theorem \ref{Planarcase} deals with the worst case among
{\em all} planar convex bodies. It may, however, also be of interest to regard restricted classes.
A natural class to consider is $\CK_o^n$, i.e., that of origin symmetric convex bodies
and hence to define
$$
\lambda_o(n)= \inf_{K\in \CK_o^n} \Lambda (K).
$$
The following remark shows that already in the plane some caution has to be exercised since,
unlike for $\lambda(2)$, the infimum need no longer be attained.

\begin{lem}\label{rem-not-attained}
The number $\lambda_o(2)$ equals $\lambda(2)$, but the infimum is not attained.
\end{lem}

\begin{proof}
In view of Theorem \ref{Planarcase} and its characterization of the equality case
it suffices to construct a sequence of planar origin symmetric convex bodies
whose $\LW$-constant tends to $\frac12$.
For $l\in (0,1)$ let $R_l$ denote a rhombus whose diagonals have length $2$ and $2l$, respectively.
By Theorem \ref{FreemanShapira}, the best frame for $R_l$ has a direction parallel to one of
its edges. Hence, by elementary calculations we obtain
$$
\Lambda(R_l)=2l\left(\frac{4l}{1+l^2}\right)^{-1}=\frac{1+l^2}{2},
$$
which tends to $\frac12$ as $l\rightarrow 0$.
\end{proof}

The results for the planar case suggest to focus on simplices in $\R^n$. We now prove Theorem~\ref{thm-simplex}.

\begin{proof}[Proof of Theorem \ref{thm-simplex}]
We start with the proof of the upper bound. Let $F=\{u_1,\dots, u_n\}$ be any frame.
For $v\in \S^{n-1}$, let $l(v)$ denote the length of the longest chord in $S_n$ parallel to $v$.
Then
\begin{equation}\label{maxchord}
    \vol_n(S_n)=\frac{l(v)}{n}\, \vol_{n-1}(S_n|v^\perp).
\end{equation}
(By \cite{Martini}, simplices are the only sets which satisfy
this formula for all $v\in \S^{n-1}$.)  Note that $S_n$ contains the convex hull $Q$ of the chords of length $l(u_i)$ parallel to $u_i$, for all $i\in [n]$. Therefore,
\begin{equation}\label{SnQinequality}
\vol_{n}(S_n)\geq \vol_n(Q)\geq  \frac{\prod_{i=1}^{n}l(u_i)}{n!}.
\end{equation}
The latter inequality can be proved by induction on $n$ on noting that
$$n\vol_n(Q)\ge l(u_i)\vol_{n-1}(Q|u_i^{\perp}),$$
for each $i\in [n]$ (an easy consequence of an inequality by Rogers and Shephard \cite{RS}). By (\ref{maxchord}) and (\ref{SnQinequality}), we obtain
$$
\Lambda (S_n;F)=\frac{\vol_{n}(S_n)^{n-1}}{\prod_{i=1}^n\vol_{n-1}(S_n|u_i^\perp)}
       = \frac{\vol_n(S_n)^{n-1}\,\prod_{i=1}^{n}l(u_i)}
       {\vol_{n}(S_n)^nn^n}
       \leq \frac{n!}{n^{n}}
$$
and hence
$$
\Lambda (S_n) \leq \frac{(n-1)!}{n^{n-1}}.
$$

We claim that equality holds for $T_n$.  Denote by $a$ a vertex of $T_n$, by $T_{n-1}$ its opposite facet, by $v$ the outward unit normal to $T_{n-1}$ and by $h_a$ the distance from $a$ to $T_{n-1}$. The line through $a$ parallel to $v$ intersects $T_{n-1}$. Therefore,
$$
  \vol_{n-1}(T_n|v^\perp)=\vol_{n-1}(T_{n-1}).
$$
For every frame $F=\{v_1, \dots, v_n\}$ with $v_1=v$, we have
$$
    \vol_{n-1}(T_n|v_i^\perp)=\frac{h_a}{n-1}\, \vol_{n-2}(T_{n-1}|v_i^\perp), \ \textrm{for} \ i=2,\dots ,n.
$$
Clearly
$$
    \vol_{n}(T_n)=\frac{h_a}{n}\, \vol_{n-1}(T_{n-1}).
$$
Consequently,
$$
    \Lambda (T_n;F) =
    \frac{\vol_{n-1}(T_{n-1})^{n-1}\,h_a^{n-1}\,(n-1)^{n-1}}
    {n^{n-1}\,\vol_{n-1}(T_{n-1})\,h_a^{n-1}\,\prod_{i=2}^{n}\vol_{n-2}(T_{n-1}|v_i^\perp)}
    = \frac{(n-1)^{n-1}}{n^{n-1}}\, \Lambda (T_{n-1};F\setminus \{v_1\}).
$$
This implies
$$
      \Lambda (T_n) \geq \frac{(n-1)^{n-1}}{n^{n-1}}\, \Lambda (T_{n-1}),
$$
since $\Lambda(T_n) \geq \max_{\{F:v_1\in F\}} \Lambda(T_n;F)$. Applying this inequality recursively and using that $\Lambda(T_2)=\frac12$
proves the claim.

Next, we prove the lower bound for $S_n$. The argument is rather similar to the one for $T_n$.  

We denote by $a_i$, $i\in [n+1]$, the vertices of $S_n$, by $F_i$ the facet opposite to $a_i$ and 
by $h_i$ the distance from $a_i$ to $F_i$. Let $F_1$ be a facet of $S_n$ with largest $(n-1)$-dimensional 
volume, and let $v$ be its outward unit normal. We assume without loss of generality that $F_1 \subset v^\perp$.

We claim that $\vol_{n-1}(S_n|v^\perp)< 2\vol_{n-1}(F_1)$.
To see this, let $p_i$ denote the projection of $a_i$ onto $v^\perp$, $i\in [n+1]$ (so $p_i=a_i$, for $i\neq 1$). 
Then $S_n|v^\perp=\conv\{p_1,\dots,p_{n+1}\}$.
Let $L$ be the simplex in $v^\perp$ each of whose facets is parallel to a facet of $F_1$ and contains 
the vertex of $F_1$ opposite this facet. (Note that $L$ is, up to translation, equal to $-nF_1$.) 
Since $\vol_{n-1}(F_1)$ is maximal, it is at least as large as the $(n-1)$-dimensional volume of the 
projection of each facet $F_i$ onto $v^\perp$, and then as large as the $(n-1)$-dimensional volume 
of the convex hull of any $n$ of the $p_i$'s. It follows that the distance from $p_1$ to any facet 
$E$ of $F_1$ is less than the height of $F_1$ relative to $E$. Therefore, $p_1$ is contained in the 
slab determined (on $v^\perp$) by the two planes parallel to $E$, and at a distance equal to the height of $F_1$ relative to $E$. Repeating the argument for all facets of $F_1$ we deduce that $p_1\in L$.

Let $g(x)=\vol_{n-1}(\conv(F_1\cup \{x\}))$, $x\in v^{\perp}$. Then $g(p_1)=\vol_{n-1}(S_n|v^{\perp})$. 
Moreover, $g(x)$ is a convex function of $x$; this is a consequence of \cite[Theorem~10.4.1]{Schneider}, 
and follows easily from its specialization \cite[p.~543]{Schneider} with $K_0=F_1$ and $K_1=\{x\}$. 
Therefore the maximum value of $g$ is attained at some vertex $w$ of $L$. Notice that $w$ is on the 
boundary of each slab considered before. Moreover, if we assume $E$ is the facet of $F_1$ parallel 
to the facet of $L$ opposite to $w$, we see that $E$ cuts $\conv(F_1\cup \{w\})$ in two simplices of 
equal $(n-1)$-dimensional volume: Hence, $g(w)=2\vol_{n-1}(F_1)$. This maximum cannot be attained or 
else a facet of $S_n$ would have a projection onto $v^\perp$ with the same volume as $F_1$. 
Thus $g(p_1)<g(w)$ and this proves the claim.

Now, for every frame $F=\{v_1, \dots, v_n\}$ with $v_1=v$, we have
$$
    \vol_{n-1}(S_n|v_i^\perp)=\frac{h_1}{n-1}\, \vol_{n-2}(F_1|v_i^\perp), \ \textrm{for} \ i=2,\dots ,n,
$$
$$
    \vol_{n}(S_n)=\frac{h_1}{n}\, \vol_{n-1}(F_1).
$$
Therefore, we deduce
$$
    \Lambda (S_n;F) >
    \frac{\vol_{n-1}(F_1)^{n-1}\,h_1^{n-1}\,(n-1)^{n-1}}
    {2\, n^{n-1}\,\vol_{n-1}(F_1)\,h_1^{n-1}\,\prod_{i=2}^{n}\vol_{n-2}(F_1|v_i^\perp)}
    = \frac{(n-1)^{n-1}}{2\,n^{n-1}}\, \Lambda (F_1;F\setminus \{v_1\}).
$$
Since $\Lambda(S_n) \geq \max_{\{F:v_1\in F\}} \Lambda(S_n;F)$, we get
$$
      \Lambda (S_n) > \frac{(n-1)^{n-1}}{2\,n^{n-1}}\, \Lambda (F_1),
$$
which, applied recursively together with $\Lambda(S_2)=\frac12$, yields the assertion.
\end{proof}

We notice that the proof of Theorem \ref{thm-simplex} makes use of few properties of the regular simplex. Indeed, we observe that the inductive step holds whenever one altitude of a simplex falls into the opposite facet. Unfortunately, for $n\geq 3$, there exist simplices with no internal altitude.

\begin{cor}\label{cor-cor}
Let $S_n$ be a simplex. Then $\Lambda(S_n)=n!/n^n$ if and only if there exists a frame $F=\{u_1,\dots,u_n\}$ 
such that the convex hull of the maximal chords of $S_n$ parallel to $u_1,\dots,u_n$, respectively, is $S_n$ itself.
\end{cor}

\begin{proof}
If $\Lambda(S_n)=n!/n^n$, then equality holds in (\ref{SnQinequality}) and so $S_n=Q$.
\end{proof}

\begin{cor}\label{cor-edgeinframe}
Let $S_n$ be a simplex. If $\Lambda(S_n)=n!/n^n$, then every best frame contains a vector that is parallel to an edge of $S_n$.
\end{cor}

\begin{proof}
Consider a best frame $F$. By Corollary \ref{cor-cor}, the vertices of $S_n$ are extremal points of some maximal chord parallel to a direction in $F$. Since we have $n+1$ vertices and $n$ chords, surely at least one chord contains two vertices, which is what we want to prove.
\end{proof}

\begin{ex}
{\rm There are simplices $S$ in $\R^3$ with $\Lambda(S)<2/9$.  For example,
consider a rectangular box with edges of length 1, 2, 4 and the simplex $S$ which is the convex hull 
of four vertices of the box no two connected by an edge of the box. 
(The edges of $S$ have length $\sqrt{5}$, $\sqrt{17}$ and $2\sqrt{5}$.) 
By Theorem~\ref{thm-simplex} with $n=3$, $\Lambda(S)\le 2/9$. Suppose that $\Lambda(S)=2/9$. Let $F=\{v_1,v_2,v_3\}$ 
be a best frame of $S$. By Corollary \ref{cor-edgeinframe}, $F$ contains a direction, say $v_1$, 
parallel to one of the edges of $S$, say $E$. Since $S$ has no pair of orthogonal edges, by 
Corollary \ref{cor-cor} we infer that the maximal chords of $S$ parallel to $v_2$ and $v_3$, 
respectively, each contain a vertex of $S$ not in $E$. Taking into account that a maximal chord 
through a vertex has to intersect the opposite facet, elementary (but tedious) calculations show that 
we get a contradiction. Hence, $\Lambda (S)<2/9$.}\qed
\end{ex}

For general convex bodies, we can give a positive lower bound for  $\lambda(n)$.  It is based on the following lemma.

\begin{lem}\label{lem-crosspolytop-in-K}
If $K\in \CK_o^n$, then there exists a rectangular crosspolytope $C$ contained in $K$ with
$$
      \vol_n(K)\leq n!\vol_n(C).
$$
\end{lem}

\begin{proof} We use induction on the dimension $n$.  For $n=2$ the statement follows by taking $C$ to be the convex hull of a diameter of $K$ and an orthogonal maximal chord. Let $n\ge 3$ and suppose that the asserted inequality holds in dimension $n-1$.  Let $v\in \S^{n-1}$ be the direction of a diameter of $K$. Since $K$ is origin symmetric, it follows from the
Brunn-Minkowski inequality (see e.g. \cite[p.~415]{Gardner}) that $V_{n-1}(K\cap v^\perp)$ is largest
among all sections $K\cap (\lambda v+v^\perp)$, $\lambda \in \R$, parallel to $v^\perp$. Hence,
\begin{equation}\label{vbound}
     \vol_n(K)\leq (\diam K) \vol_{n-1}(K\cap v^\perp).
\end{equation}
By the induction hypothesis, there exists an $(n-1)$-dimensional rectangular crosspolytope $C'$ in $K\cap v^\perp$ with
\begin{equation}\label{cbound}
     \vol_{n-1}(K\cap v^\perp)\leq (n-1)!\vol_{n-1}(C').
\end{equation}
Let
$$
C=\conv \left(\frac12 (\diam K)[-v,v]\cup C'\right)
$$
and note that $C$ is a rectangular crosspolytope.  Using (\ref{vbound}) and (\ref{cbound}), we obtain
$$
    \vol_n(K)\leq (n-1)!(\diam K) \vol_{n-1}(C')=n!\vol_n(C),
$$
which completes the proof.
\end{proof}

Now we can prove Theorem~\ref{thm-lower bound}.

\begin{proof}[Proof of Theorem \ref{thm-lower bound}]
It suffices to show that, for every $K\in \CK^n$, there exists a frame $F=\{u_1, \dots, u_n\}$ such that
$$
\Lambda(K;F)\geq \frac{\binom{2n}{n}}{(2n)^{n}}.
$$
Let $F$ be the frame that consists of the directions of the axes of a rectangular crosspolytope $C$ of maximum volume
contained in $\PicK$.
By (\ref{polar}), (\ref{projbody}), (\ref{prop-zang}) and Lemma \ref{lem-crosspolytop-in-K}, we have
$$
\Lambda(K;F)= \frac{\vol_n(K)^{n-1}}{\prod_{i=1}^n h_{\PiK}(u_i)}
\geq  \frac{\binom{2n}{n}}{2^n n^n}\, \frac{\prod_{i=1}^n 2\rho_{\PicK}(u_i)} {\vol_n(\PicK)}
= \frac{\binom{2n}{n}}{2^n n^n}\, \frac{n!\, \vol_n(C)} {\vol_n(\PicK)}
\geq \frac{\binom{2n}{n}}{(2n)^n},
$$
which proves the assertion.
\end{proof}

Since the case of the standard cube is trivial and that of the simplex is investigated in Theorem
\ref{thm-simplex}, it might be worthwhile to consider the last of the three regular polytopes that
exist in any dimension, the regular crosspolytope $C_n=\conv\{\pm e_1,\ldots, \pm e_n\}$.

\begin{thm}\label{thm-C_n}
$$
\Lambda(C_n)\ge \left\{\begin{array}{cl}
\DS \frac{(n-1)! \, \, 2^{\frac{n}{2}}}{n^{n-1}} & \mbox{for $n$ even};\\[.4cm]
\DS \frac{(n-1)! \, \, 2^{\frac{n-1}{2}}}{n^{n-1}} & \mbox{for $n$ odd}.\\
\end{array}\right..
$$
\end{thm}

\begin{proof} First note that
$$
\vol_n(C_n)= \frac{2^n}{n!}, \qquad \vol_{n-1}(C_n|e_1^{\perp})= \frac{2^{n-1}}{(n-1)!},
$$
and that all the facets of $C_n$ are regular simplices $T_{n-1}$, with
$$
 \vol_{n-1}(T_{n-1}) = \frac{\sqrt{n}}{(n-1)!}.
$$

Also, if $u$ is the direction of an edge of $C_n$, e.g.,
$$
u=\frac{1}{\sqrt{2}}(e_1-e_2)
$$
and $w_j$, $j\in [2^n]$, denote the $2^n$ facet outward unit normals $(\pm 1/\sqrt{n},\ldots,\pm 1/\sqrt{n})$ of $C_n$, we have
$$
\vol_{n-1}(C_n|u^{\perp})= \frac{1}{2} \vol_{n-1}(T_{n-1})\sum_{j=1}^{2^n} |u^Tw_j|
= \frac{1}{2\sqrt{2}(n-1)!} 2^{n-2} (2+2)= \frac{1}{\sqrt{2}(n-1)!} 2^{n-1}.
$$

Now, let $n$ be even. Then, for $i\in [n/2]$, the vectors
$$
u_{2i-1}=\frac{1}{\sqrt{2}} (0,\ldots,0,1,1,0,\ldots,0)^T,\qquad
u_{2i}=\frac{1}{\sqrt{2}} (0,\ldots,0,1,-1,0,\ldots,0)^T
$$
(where the first nonzero entries are in the $(2i-1)$th position),
corresponding to directions of edges of $C_n$, form a frame $F$.
Therefore,
$$
\Lambda(C_n)\ge\Lambda(C_n;F)= \frac{\vol_n(C_n)^{n-1}}{\vol_{n-1}(C_n|u^{\perp})^n}
= \DS \frac{(n-1)! \, \, 2^{\frac{n}{2}}}{n^{n-1}}.
$$
If $n$ is odd, we use $n-1$ projections along edges and one projection along an orthogonal coordinate direction
$e$ and obtain
$$
\Lambda(C_n)\ge \frac{\vol_n(C_n)^{n-1}}{\vol_{n-1}(C_n|u^{\perp})^{n-1}\vol_{n-1}(C_n|e^{\perp})}
= \DS \frac{(n-1)! \, \, 2^{\frac{n-1}{2}}}{n^{n-1}},
$$
which completes the proof.
\end{proof}

\begin{con}\label{conjecture-C_n-intro}
Let $n$ be even. Then
$$
\Lambda(C_n)= \DS \frac{(n-1)! \, \, 2^{\frac{n}{2}}}{n^{n-1}}.
$$
\end{con}

Note that Conjecture \ref{conjecture-C_n-intro} is trivially correct for $n=2$, follows
from \cite{martini-weissbach-84} for $n=4$, and for general even $n$ from the conjecture of \cite[p.167]{martini-weissbach-84}
that the orthogonal projections of $C_n$ of minimal volume occur in the direction of edges
of $C_n$. We remark that (when specified appropriately) the problems of finding
maximal or minimal projections of polytopes on hyperplanes are $\NP$-hard even for very
special classes of polytopes; see
\cite{burger-gritzmann-00} and \cite{burger-gritzmann-klee-96}.

\section{Algorithmic issues}\label{sec-algorithms}

We will now deal with the reverse Loomis--Whitney problem from a computational point of view.
In order to be able to employ the binary {\em Turing machine model} we will restrict our
considerations to rational polytopes.

\subsection{Volume computation and the case of variable dimension}
Since the projection averages and the $\LW$-constant involve volume computations
we begin with a brief account of the algorithmic properties of the task to
compute or approximate the volume of given polytopes.

Let $P$ and $Q$ be polytopes in $\R^n$. We say that $P$ is an {\em $\CH$-polytope} and write
$P=(n,m,A,b)$, if $P$ is specified by $n,m\in \N$, $A\in \Q^{m\times n}$ and $b\in \Q^m$, such that $P=\{x\in \R^n: Ax\le b\}$.

Similarly, $Q$ is called a {\em $\CV$-polytope}, written $Q=(n,k,V)$, if $Q$
is specified by $n,k\in \N$, $V\subset \Q^{n}$ with $k=|V|$ and $Q=\conv (V)$.

As it is well-known, each $\CH$-polytope admits a $\CV$-presentation, each
$\CV$-polytope admits an $\CH$-presentation, and one presentation can be derived
from the other in polynomial time if the dimension $n$ is {\em fixed}.
As, for instance, the standard $n$-cube has $2n$ facets but $2^n$ vertices, and the standard $n$-cross-polytope has $2^n$ facets but $2n$ vertices, we have, however, to distinguish the different kinds of representations if the dimension is part of the input.

In the binary {\em Turing machine} model, the {\em size} of the input is measured as the length $\size(P)$ of the binary encoding needed to present the input data. The running time of an algorithm is defined in terms of the number of operations of its corresponding Turing machine; see e.g. \cite{GJ-79}.

In the sequel, {\sc $\CH$-Volume} and {\sc $\CV$-Volume} will denote the following decision versions:
\begin{quote}
Given an $\CH$-polytope or a $\CV$-polytope $P$, respectively, and $\nu\in \Q$ with
$\nu \ge 0$, decide whether $\vol_n(P) \le \nu$.
\end{quote}
In the randomized case, there are additional positive rationals $\beta$ and $\epsilon$,
and the task is to determine a positive rational random variable $\nu$ such that
$$
\prob \left\{\left| \frac{\nu}{\vol_n(P)}-1\right| \le \epsilon\right\}
\ge 1-\beta.
$$
The following proposition summarizes complexity results that are relevant for
our purpose. The non trivial results are due to \cite{dyer-frieze-88}, \cite{dyer-frieze-89}, \cite{khachiyan-88}, \cite{khachiyan-89} and \cite{lawrence-91}; see \cite{gritzmann-klee-94b} and \cite{gritzmann-klee-17} for more details, further references to related results and extensions.

\begin{prop}\label{prop-volume}
\begin{enumerate}
\item If the dimension is fixed and $P$ is a $\CV$- or an $\CH$-polytope, the volume of $P$ can be computed in polynomial time.
\item If $P$ is a $\CV$-polytope, then the binary size of $\vol_n(P)$ is bounded above by a polynomial in the input size. In general, the binary size of $\vol_n(P)$ for $\CH$-polytopes $P$ is not bounded above by a polynomial in the input size.
\item {\sc $\CH$-Volume} and {\sc $\CV$-Volume} are  $\# \P$-hard. If fact, the problem of computing the volume is $\#\P$-hard both, for the intersection of the $\CH$-unit cube with a rational halfspace and for the convex hull of the regular $\CV$-crosspolytope and an additional integer vector.
\item There is a randomized algorithm for volume computation for $\CV$- and $\CH$-polytopes which runs in time that is polynomial in the input size and in $1/\epsilon$.
\end{enumerate}
\end{prop}

The $\LW$-ratio for a given polytope and a given frame is the ratio of two terms that
involve volume computation. Hence, it might seem natural to expect that the
$\#\P$-hardness results for volume computation imply intractability results for
$\LW$-constant computation.

To be more specific, let in the sequel
$a_1,\ldots,a_m \in \Q^{n}\setminus \{0\}$, $\beta_1,\ldots,\beta_m\in \Q$, such that
$$
P=\{x\in \R^n: a_1^Tx\le \beta_1, \ldots, a_m^Tx\le \beta_m\}
$$
is an $n$-dimensional polytope in $\R^n$ with $m$ facets
$F_1, \dots, F_m$. Since our problem is invariant under translations we can assume that
$\beta_1,\ldots,\beta_m > 0$. Then $\beta_j/ \|a_j\|$ is the distance
of $F_j$ from the origin.
While $\vol_n(P)\in \Q$, the facet volumes $\vol_{n-1}(F_j)$ need not be rational.
However, since
$$
\vol_{n-1}(F_j) = \frac{n \|a_j\|}{\beta_j}\vol_n\bigl(\conv(\{0\}\cup F_j)\bigr),
$$
the {\em normalized facet volumes}
$$
\omega_j=\frac{\vol_{n-1}(F_j)}{\|a_j\|}, \qquad j\in [m],
$$
are all rational. Using (\ref{eq-zonotope}), we then have the following identity.

\begin{rem}\label{rem-LW-ratio-poly}
$$
\left(\frac{1}{n} \sum_{j=1}^m  \beta_j\omega_j\right)^{n-1} = \Lambda(P;F) \cdot
\prod_{i=1}^n \left(\frac{1}{2} \sum_{j=1}^m \omega_j |u_i^Ta_j|\right) .
$$
\end{rem}

Therefore $\Lambda(P;F)$ is the ratio of terms involving the normalized facet volumes of $P$.
Hence, from the knowledge of $\Lambda(P;F_k)$ for a suitable finite set of different
frames $F_k\in \CF^n$ one can in principle compute all normalized facet volumes
(through the corresponding system of multivariate polynomial equations along the lines of \cite{DGH98}).
Hence $\vol_n(P)$ can be approximated up to any given precision $2^{-p}$, $p\in \N$.
This indicates that the computation of $\Lambda(P;F_k)$ for such frames should be $\#\P$-hard.

The situation seems, however, much less obvious if we just had an efficient algorithm for computing
(or approximating) $\Lambda(P)$. It is, in fact, not clear at all how we could actually use
the knowlege of $\Lambda(P)$ for volume computation (or to solve any other $\#\P$-hard problem). So it is conceivable
that the task of computation the $\LW$-constant of $\CV$- or $\CH$-polytopes is easier than volume computation.
We conjecture, however, that this is not the case.

\begin{conj}\label{conj-hardness}
The problem, given a $\CV$-polytope $P$ (or an $\CH$-polytope, respectively) and a precision $p\in \N$,
approximate $\Lambda(P)$ up to $p$ binary digits, is $\#\P$-hard.
\end{conj}

It is also conceivable that $\LW$-constant computation is actually harder than volume computation in
that not even randomization helps. Similarly to radii-computations \cite{bgkkls98}, \cite{bgkkls01} there may
exist a sequence of convex bodies $K_n\in \CK^n$ for which the set of all frames
that produce good approximations of $\Lambda(K_n)$ has measure that decreases too quickly
in the dimension $n$.

Getting back to the deterministic case, suppose now that we had an efficient algorithm at hand that
would compute (or closely approximate) the volume of polytopes. Could we use this algorithm as a
subroutine to compute $\Lambda(P)$ efficiently?
Recall that by (\ref{eq-zonotope}) and Lemma \ref{rem-frame-box} the computation of
$\Lambda(P)$ could then be achieved by finding a rectangular box of minimal volume containing the zonotope
$$
\PiP = \frac{1}{2} \sum_{j=1}^m \omega_j  [-a_j, a_j].
$$
As it turns out every zonotope is actually the projection body of a polytope.

\begin{lem}\label{lem-box-zonotope}
Let $a_1,\ldots,a_m \in \R^n\setminus \{0\}$ spanning $\R^n$ and set
$$Z= \frac{1}{2}\sum_{j=1}^m[-a_j, a_j].$$
Then there exists a polytope $Q$ with outer facet normals $\pm a_1,\ldots,\pm a_m$ such that
$Z=\PiQ$.
\end{lem}

\begin{proof} Let
$$
b_j= \frac{a_j}{\|a_j\|}, \qquad b_{n+j}= - \frac{a_j}{\|a_j\|}, \qquad
\gamma_j=\gamma_{n+j}= \frac{1}{2} \|a_j\|,\qquad j\in [m].
$$
Then
$$
\sum_{j=1}^{2m} \gamma_j b_j =0,
$$
and we apply Minkowski's Theorem \cite{minkowski-03}; compare also \cite{minkowski-11},
and see \cite{gritzmann-hufnagel-99} for an algorithmic version.
Hence there exists a polytope $Q$, unique up to translation, with facet outward unit normals
$b_1,\ldots,b_{2m}$ and corresponding facet volumes $\gamma_1,\ldots,\gamma_{2m}$, and we have
$$
\PiQ= \frac{1}{2} \sum_{j=1}^{2m} \gamma_j [-b_j, b_j]
= \frac{1}{4} \sum_{j=1}^{m} [-a_j, a_j]+\frac{1}{4} \sum_{j=1}^{m} [a_j,-a_j]
= \frac{1}{2} \sum_{j=1}^{m} [-a_j,a_j]=Z.
$$
\end{proof}

Let us point out that Lemma \ref{lem-box-zonotope} does not claim that $Q$ is a rational polytope.
Nevertheless we believe that the difficulty of the computation of $\Lambda(P)$ is not caused
by that of volume computation alone but is also due to the intractability of the following problem.
\begin{quote} {\sc BoxContZonotope}: Given $a_1,\ldots,a_m\in \Q^n$ and $\nu \in [0,\infty[\cap \Q$;
decide whether there exists a rectangular box $B$ such that
$$
\sum_{j=1}^m [-a_j, a_j] \subset B, \qquad \vol_n(B) \le \nu.
$$
\end{quote}

\begin{conj}\label{conj-box-zonozope}
The problem {\sc BoxContZonotope} is $\NP$-hard.
\end{conj}

Note that Conjecture \ref{conj-box-zonozope} is related to a conjecture of \cite{Me90}
that the following problem is $\NP$-complete: Given a point set $V$ in $\R^n$, does
there exists a rigid motion $\varphi$ such that $V$ is contained in $\varphi([-1,1]^n)$?

\subsection{Fixed dimension}
Now we turn to the case of fixed dimensions, i.e., we assume that $n$ is not part of the input but
constant. Using the same notation as before, let $P$ be a rational polytope with facets
$F_1, \dots, F_m$. Of course, as $n$ is constant, $\vol_n(P)$, and the normalized facet
volumes $\omega_1,\ldots,\omega_m$ can be computed in polynomial time.

With respect to exact computations the only parts of Remark \ref{rem-LW-ratio-poly} that
require attention are the terms $|u_i^Ta_j|$.
Note that the structural result of Theorem \ref{thm-facets} that relates best frames to the
normals of $P$ does not imply that there always exists
a best frame that (up to norm computations) can be finitely encoded.
More precisely, we have the following problem.

\begin{probl}\label{prob-rational-vectors}
Given a rational polytope $P$, do there exist integer vectors $v_1,\ldots,v_n\in \Z^n\setminus \{0\}$ such that
$$
\left\{\frac{v_1}{\|v_1\|}, \ldots, \frac{v_n}{\|v_n\|}\right\}
$$
is a best frame for $P$?
\end{probl}

Even if Problem \ref{prob-rational-vectors} would admit an affirmative answer
the evaluation of $\Lambda(P;F)$ will in general still need to involve approximations
of norms. In the following we will therefore resort to suitably specified approximations.

First we give additive error bounds for projection average computation after a suitable scaling of
$P$. Recall that $\Lambda$ is invariant under scaling. More precisely, we assume that
$$
1 \le \area(P)=\sum_{j=1}^m  \vol_{n-1}(F_j)= \sum_{j=1}^m \omega_j\|a_j\|\le 1+ \frac{1}{n}.
$$
Such a scaling can be done in polynomial time.

We begin with a technical lemma which shows that a tight rational approximation of the vectors of a frame $F$
leads to a tight approximation of $\Lambda(P;F)$.

\begin{lem}\label{lem-error}
Let $1\le \area(P)\le 1+1/n$, $\tau \in (0,1]$ and set
$$
\rho =\frac{2^n}{e n^n} \, \tau.
$$
Further, let
$F=\{u_1,\ldots,u_n\}\in \CF^n$ and $w_1,\ldots,w_n \in \Q^n$
such that
$$
\|w_i -u_i\| \le  \rho\ ,  \qquad i\in [n].
$$
Then
$$
\left| \Psi(P;F) -  \frac{1}{2^n} \prod_{i=1}^n \left(\sum_{j=1}^m \omega_j  |w_i^Ta_j|\right) \right|
\le \tau.
$$
\end{lem}

\begin{proof}
Let $z_i=w_i -u_i$, $i\in [n]$. Then, using the abbreviation $\alpha=1+1/n$,
we have for $i\in [n]$ und $j\in [m]$
$$
\omega_j |w_i^Ta_j|\le \omega_j |u_i^Ta_j|+\omega_j |z_i^Ta_j|\le \omega_j |u_i^Ta_j| + \rho \omega_j \|a_j\|,
$$
hence,
$$
\sum_{j=1}^m \omega_j |w_i^Ta_j| \le \sum_{j=1}^m \omega_j |u_i^Ta_j| + \rho \sum_{j=1}^m \omega_j \|a_j\|
\le \rho \alpha + \sum_{j=1}^m \omega_j |u_i^Ta_j|,
$$
and therefore
$$
 \DS \prod_{i=1}^n \left(\sum_{j=1}^m \omega_j |w_i^Ta_j|\right)
\le  \prod_{i=1}^n \left(\rho \alpha +\sum_{j=1}^m \omega_j |u_i^Ta_j|\right).
$$
Setting
$$
\xi= \rho \alpha, \quad \eta_i= \sum_{j=1}^m \omega_j |u_i^Ta_j| \ ,\qquad i\in [n],
$$
the right hand side can be expressed in terms of the elementary symmetric polynomials
in $\eta_1,\ldots,\eta_n$, i.e.,
$$
\prod_{i=1}^n (\xi+ \eta_i) = \xi^n + \sigma_1\xi^{n-1}+  \sigma_2 \xi^{n-2}+\ldots + \sigma_{n-1} \xi+\sigma_n
$$
with
$$
\sigma_k= \sum_{S\subset [n]\atop |S|=k} \prod_{i\in S} \eta_i \ , \qquad k\in [n].
$$
Since
$$
\eta_i = \sum_{j=1}^m \omega_j |u_i^Ta_j| \le \sum_{j=1}^m \omega_j \|a_j\| \le \alpha
$$
it follows that
$$
\sigma_k\le \sum_{S\subset [n]\atop |S|=k}\alpha^k = {n \choose k}\alpha^k \le (n \alpha)^k.
$$
Thus, with $\sigma_0=1$ and $\rho \le 1$ we have
$$
\begin{aligned}
 \prod_{i=1}^n (\xi+ \eta_i)
& = \sigma_n + \xi \sum_{i=0}^{n-1} \sigma_i \xi^{n-1-i}
 \le  \sigma_n + \xi \sum_{i=0}^{n-1} (n \alpha)^{i} (\rho \alpha)^{n-i-1}\\
& = \sigma_n + \alpha^n \rho \sum_{i=0}^{n-1} n ^{i} \rho ^{n-i-1}
\le  \sigma_n +\alpha^n n^n \rho.
\end{aligned}
$$
Therefore, we obtain
$$
 \DS \frac{1}{2^n} \prod_{i=1}^n \left(\sum_{j=1}^m \omega_j |w_i^Ta_j|\right)
 \le  \Psi(P;F) + \frac{\alpha^n n^n}{2^n}\, \rho \le  \Psi(P;F) + \frac{e n^n}{2^n}\, \rho
 = \Psi(P;F) + \tau,
$$
which proves the first part of the assertion.

Similarly, for $i\in [n]$ und $j\in [m]$,
$$
\omega_j|w_i^Ta_j| \ge  \omega_j|u_i^Ta_j|  - \omega_j|z_i^Ta_j|
\ge \omega_j|u_i^Ta_j| - \rho \omega_j \|a_j\|
$$
and
$$
\prod_{i=1}^n (-\xi+ \eta_i) = \sigma_n - \xi \sum_{i=0}^{n-1} \sigma_i (-\xi)^{n-1-i}
\ge \sigma_n - \xi \sum_{i=0}^{n-1} \sigma_i \xi^{n-1-i}.
$$
Hence, the same arguments as before yield
$$
 \DS \frac{1}{2^n} \prod_{i=1}^n \left(\sum_{j=1}^m \omega_j |w_i^Ta_j|\right)
 \ge  \frac{1}{2^n}  \prod_{i=1}^n \left(-\rho \alpha  +\sum_{j=1}^m \omega_j |u_i^Ta_j|\right)
\ge  \Psi(P;F) - \frac{e n^n}{2^n}\, \rho = \Psi(P;F) - \tau,
$$
which proves the assertion.
\end{proof}

Lemma \ref{lem-error} shows that it is enough to approximate the vectors of a best frame
by rational vectors up to some distance $O(\tau)$ in order to obtain
the $\LW$-constant of $P$ up to a given precision $\tau$.

The algorithm for approximating $\LW$-constants will now be based on the structural results of
Theorem \ref{thm-facets} augmented  by a discrete approximate sampling of the set of all frames
in order to take care of the remaining degrees of freedom.

Recall that according to Theorem \ref{thm-facets}, with at most one exception, the vectors of a
best frame are parallel to facets of $P$.
Thus in addition to the normalization and the orthogonality constraints we have
the conditions
$$
\exists j\in [m]: a_j^Tu_1=0, \ldots, \exists j\in [m]: a_j^Tu_{n-1}=0.\
$$
A simple count shows that this still leaves $\frac{1}{2}(n-1)(n-2)$ remaining
degrees of freedom for each of the
$$
\binom{m+n-2}{m-1}=O(m^{n-1})
$$
(minimal) choices for the $u_i$ orthogonal to some $a_j$. Each such potential choice is characterized
by a subset $R$ of $[n]\times [m]$ of those pairs $(i,j)$ for which $a_j^Tu_i=0$. Let
$\CR$ denote the set of all such $R$.

The algorithm will perform $O(m^{n-1})$ steps of exhaustive search involving the
remaining degrees of freedom.
The candidates for near-best frames will be constructed successively.
In the $i$th step we construct a set $W_i$ of rational vectors $w_i$ of length close to $1$
that are orthogonal to $w_1,\ldots,w_{i-1}$ and satisfy $a_j^Tw_i=0$, $j\in J_i$, for some
$J_i \subset [m]$.
By Theorem \ref{thm-facets} we can assume that $J_i\ne \emptyset$ for $i\in [n-1]$.

The linear conditions in step $i$ define a linear subspace $X_i$ of $\R^n$. We use the following
result from algorithmic linear algebra which is simply based on Gauss elimination and Gram-Schmidt orthogonalization; see e.g. \cite{schrijver}.

\begin{rem}\label{rem-basis}
Let $X_i$ be a $q$-dimensional subspace of $\R^n$ given as the solution of a system of linear equations with
rational coefficients. Then there exist linearly independent vectors $x_1,\ldots, x_q \in X_i\cap \Q^n$
such that
$$
x_k^Tx_l=0 \ , \quad  k,l\in [q],\,k<l , \qquad 1\le \|x_k\| \le 2 \ , \quad k\in [q].
$$
Such vectors can be found in polynomial time; they will be referred to as a
{\em near-normal orthogonal basis} of $X_i$.
\end{rem}

Next we construct the sets $W_i$.

\begin{lem}\label{lem-sampling}
Let $\tau\in (0,1]\cap \Q$ and $\rho =\tau 2^nn^{-n}e^{-1}$,
as in Lemma \ref{lem-error}, and set $\alpha=\rho/(2n)$.
Further, let $x_1,\ldots, x_q \in \Q^n$
be a near-normal orthogonal basis of $X_i$, and set
$$
W_i=\bigl\{w= \sum_{k=1}^q \tau_k x_k: \tau_k \in \alpha \Z, \,
w\in (1+\rho) \B^n, \, w\not\in \inte(\B^n) \bigr\}.
$$
Then
$$
|W_i| \le \left(\frac{e n^{n+1}}{2^{n-3}\tau}\right)^q.
$$
Further, for every vector $u\in X_i\cap  \S^{n-1}$ there exists a $w\in W_i$ such that
$$
 \|u-w\| \le \rho.
$$
\end{lem}

\begin{proof} Clearly, $|W_i|$ is bounded from above by the number of points of the form
$$
\sum_{k=1}^q \bigl(\alpha \Z \cap (1+\rho)[-1,1]\bigr) \frac{x_k}{\|x_k\|}.
$$
Hence
$$
|W_i| \le \left( 2 \frac{1+\rho}{\alpha} +1 \right)^q=
\left(4n\bigl(1+\tau^{-1} 2^{-n}n^n e\bigr)+1 \right)^q \le \left(\frac{8 e n^{n+1}}{\tau 2^n}\right)^q
= \left(\frac{e n^{n+1}}{\tau2^{n-3}}\right)^q.
$$

Now, for $u\in X_i\cap  \S^{n-1}$ and $k\in [q]$ let
$\sigma_{k} \in \{-1,1\}$ such that
$\bigl(u+ \sigma_{k}tx_k\bigr)\cap \B^n=\emptyset$, for every positive $t$.
Then,
$$
u + \sum_{k=1}^q \sigma_{k} [0,\alpha] x_k \subset \bigl((1+\rho) \B^n\bigr) \setminus \inte(\B^n).
$$
By construction, the box in $X_i$ on the left-hand side contains a point $w \in W_i$,
and we have
$$
 \|u-w\|\le 2 \sqrt{q}\alpha  \le \rho,
$$
which completes the proof.
\end{proof}

The elements
$$
W=\{w_1,\dots,w_m\}\in \CW(R)=W_1 \times W_2 \times \ldots \times W_{n-1}\times W_n
$$
will be called {\em pseudo frames} and the corresponding term
$$
\Psi(P;W) =\frac{1}{2^n}\prod_{i=1}^n \left(\sum_{j=1}^m \omega_j  |w_i^Ta_j|\right)
$$
will be referred to as {\em pseudo average}.

The algorithm {\sc Structured Search} for approximating $\Psi(P)$ (and $\Lambda(P)$) in fixed dimension
simply computes for each $R\in \CR$ and each pseudo frame $W\in \CW(R)$
the pseudo average $\Psi(P;W)$ and takes the minimum.

The number of pseudo average computations in each step depends on the specific set
$R\in \CR$ under consideration. As general bounds, we have
$$
\dim (X_i) \le n-i \ , \qquad i\in [n-1], \qquad \dim (X_n) = 1,
$$
from which we obtain the following result.

\begin{rem}\label{rem-number-of-searches}
$$
|\CW (R)| \le \left(\frac{en^{n+1}}{\tau 2^{n-3}}\right)^{\frac{(n-1)n+2}{2}}.
$$
\end{rem}

Note that it would be enough to use only one representative of the
class $\{\pm w_1,\ldots,\pm w_n\}$ for each $\{w_1,\dots,w_m\}\in \CW(R)$.
Hence we could reduce the $|\CW(R)|$ by a factor of $2^n$.

\begin{thm}\label{thm-additive-error}
Let $1\le \area(P)\le 1+1/n$. For a given bound $\tau \in (0,1] \cap \Q$
we can compute a number $\psi$ with
$$
\Psi(P) \le \psi \le \Psi(P)+ \tau
$$
in time that is polynomial in $\size(P)$ and $1/\tau$.
\end{thm}

\begin{proof}
We compute for each $O(m^{n-1})$ choices of $R\in \CR$ the pseudo averages of $P$
for all $O\bigl(\tau^{-\frac{(n-1)n+2}{2}}\bigr)$ elements in $\CW (R)$ and take
$\psi$ as the minimum.
It then follows from Lemmas \ref{lem-error} and \ref{lem-sampling}
that $\psi$ differs from $\Psi(P)$ only by an additive term of $\tau$.
\end{proof}

Recall that the above additive error depends on the scaling. The following result
uses the fact that $\Lambda(P)$ is invariant under scaling. The running time
of the algorithm depends, however, on the isoperimetric ratio of $P$.

\begin{cor}[see Theorem \ref{thm-poly-time-fixed-dim}]\label{thm-poly-time-fixed}
Let $n$ be fixed. Then, given $\nu \in \Q$ with $\nu >0$, an $\CH$-polytope $P$
with $\nu \le \iso(P)$, and a bound $\delta \in ]0,1[ \cap \Q$,
a number $\Lambda$ with
$$
\Lambda \le \Lambda(P) \le (1+\delta) \Lambda
$$
can be computed in time that is polynomial in $\size(P)$, $1/\delta$ and $1/\nu$.
\end{cor}

\begin{proof}
In polynomial time we can compute a dilatation factor $\sigma$ such that
$1\le \area(\sigma P)\le 1+1/n$. Since $\Lambda(P)$ is invariant under scaling
we will in the following assume that $P$ itself already lies within these bounds, i.e.,
$1\le \area(P)\le 1+1/n$.
Then we apply Theorem \ref{thm-additive-error} to compute for
$$
\tau =\min \{1,\nu \cdot \delta\}
$$
in polynomial time a number $\psi$ with
$$
\Psi(P) \le \psi \le \Psi(P)+ \tau,
$$
and set $\Lambda= \vol_n(P)^{n-1}/\psi$.
Hence,
$$
\Lambda(P) = \frac{\vol_n(P)^{n-1}}{\Psi(P)}=\frac{\psi}{\Psi(P)}\, \Lambda
\le \left(1+\frac{\tau}{\Psi(P)}\right)\, \Lambda.
$$
Using the classical Loomis-Whitney inequality \eqref{LoomisWhitneyineq}
we obtain
$$
\Lambda(P) \le \left(1+\frac{\tau}{\vol_n(P)^{n-1}}\right)\, \Lambda =
\left(1+\frac{\tau}{\area(P)^{n}\iso(P)}\right)\, \Lambda.
$$
Since by our scaling $\area(P)^n\geq 1$, this implies that
$$
\Lambda(P) \le \left(1+\frac{\tau}{\iso(P)}\right)\, \Lambda
\le \left(1+\frac{\nu }{\iso(P)}\cdot \delta \right)\, \Lambda  \leq (1+\delta)\, \Lambda.
$$
On the other hand,
$$
\Lambda(P) =  \frac{\vol_n(P)^{n-1}}{\Psi(P)} \geq  \frac{\vol_n(P)^{n-1}}{\psi}=\Lambda , $$
which proves the assertion.
\end{proof}

Note that, by Proposition \ref{prop-volume}, $\vol_n(P)$ and $\omega_1,\ldots, \omega_m$ can be computed in polynomial time.
Also $\|a_1\|, \ldots, \|a_m\|$ can be approximated efficiently. Hence we can compute an appropriate lower bound $\nu$
for $\iso(P)$ in polynomial time. This bound $\nu$ enters the running time of the algorithm, however, directly rather
than its binary size. Hence the algorithm is in general only pseudopolynomial in the input data and in
$1/\delta$.

While the structural results of Section \ref{sec-face-structure} were used in {\sc Structured Search}
to reduce the number of projection average computations, the characterization of
Theorem \ref{thm-facets} is too weak to produce a combinatorial algorithm.
(Also, the situation of the simplex shows that it is not likely that a much stronger
characterization exists.) This is in accordance with a result of  \cite{O'Rourke85} which shows that for $n=3$ the
one degree of freedom is governed by a polynomial of degree $6$ which can be used to devise an
$O(m^3)$ algorithm in the real RAM-model in that case.

\section{Related functionals}\label{sec-functionals}

We now give some additional results which involve minimal rectangular boxes containing
a given convex body. We begin with the functional $\Phi$ introduced in Section \ref{sec-face-structure}.
While for $K\in \CK^2$ we have $\Lambda(K)=\Phi(K)$, the following result gives an
inequality for general $n$.

\begin{lem}\label{LambdaPhi}
For $K\in \CK^n$, we have
$$
    \Lambda(K)^{1/(n-1)}\geq \Phi(K)\, .
$$
\end{lem}

\begin{proof}
Let $F=\{u_1,\ldots,u_n\}$ be a frame such that $\Phi(K)=\vol_n(K)/\vol_n(B(K;F))$. Then
$$
\begin{aligned}
\Lambda(K)^{1/{n-1}} &\geq  \frac{\vol_n(K)}{\prod_{i=1}^n\vol_{n-1}(K|u_i^\perp)^\frac1{n-1}}
\geq \frac{\vol_n(K)}{\prod_{i=1}^n\vol_{n-1}(B(K;F)|u_i^\perp)^\frac1{n-1}}\\
& =
\frac{\vol_n(K)}{\vol_n(B(K;F))}= \Phi(K),
\end{aligned}
$$
which concludes the proof.
\end{proof}

A lower estimate of $\Phi(K)$ is provided by the following lemma.

\begin{lem}\label{lowerboundPhi}
For every $K\in\CK^n$,
$$
   \Phi(K)\geq 1/n!.
$$
\end{lem}

\begin{proof}  The proof proceeds by induction on $n$.  
For $n=2$, the inequality is a consequence of the fact that $\Lambda(2)=1/2$. 

Let the inequality be true now for every body in $\CK^{n-1}$,
let $K$ be a convex body in $\R^n$, let $v\in \S^{n-1}$ be parallel
to a diameter of $K$, and let $D\subset K$ be a segement of length $\diam K$ parallel to $v$. 
Further, let $\hat B$ be a rectangular box in $v^\perp$ containing
the orthogonal projection of $K$ onto $v^\perp$ such that
$\vol_{n-1}(\hat B)\leq (n-1)!\vol_{n-1}(K|v^\perp)$.
We define the rectangular box $B$ as the Minkowski sum of $\hat B$
and $D$. Then $B$ contains $K$ and
$$
 \vol_n(K)\geq\frac{1}{n}\,(\diam K)\vol_{n-1}(K|v^\perp)
 \geq \frac1n\, (\diam K)\frac{\vol_{n-1}(\hat B)}{(n-1)!}=\frac{\vol_n(B)}{n!}.
$$
\end{proof}

Replacing the volume $\vol_n$ by an {\em intrinsic volume} $V_i$ (see, e.g., \cite[p.~213]{Schneider})
we obtain the functionals
$$
     \Phi_i(K)=\max_{F\in \CF^n}\frac{V_i(K)}{V_i(B(K;F))}\ , \qquad i\in [n].
$$
Recall that $\vol_n=V_n$, hence $\Phi_n=\Phi$. Also, up to a constant, $V_1$ is the {\em mean width}, and $S=2V_{n-1}$.  
Let us point out that,
rather dealing with extrema of $V_i(B(K;F))$ as a function of $F$, Chakerian \cite{Chakerian} and 
Schneider \cite{Schneider72} considered the mean value of such a function.

For $\Phi_1$ we can give additional results.
We begin with the analogue of Proposition \ref{FreemanShapira} and Theorem \ref{GeneralizedO'Rourke}.

\begin{thm}\label{FreemanShapira1}
If $P$ is a planar convex polygon, then all rectangles of minimal perimeter containing $P$
have at least one edge parallel to an edge of $P$.
\end{thm}

\begin{proof} Let $u(\theta)=(\cos\theta, \sin\theta)$. Without loss of generality, suppose that the origin lies in the interior of $P$ and consider the frames $F\in \CF^ 2$ of the form $F(\theta)=\{u(\theta), u(\theta+\pi/2)\}$, with $\theta \in [0,\pi)$.  Note that the perimeter of $B(P;F(\theta))$ is given by the function
$$
    \varphi(\theta)=2h_{\DP}(u(\theta))+2h_{\DP}(u(\theta+\pi/2)),
$$
where $\DP =P+(-P)$ is the {\em difference body} of $P$.
We want to show that the minimum of $\varphi$ is attained at some $\theta$ so that
$u(\theta)$ or $u(\theta+\pi/2)$ is parallel to some edge of $P$.
Let $p_1, p_2, \dots, p_s$ be the vertices of $\DP$. Then, of course, for $u\in \S^1$,
$$
         h_{\DP}(u)=\max_{i\in [s]}u^Tp_i.
$$
Therefore, $\varphi$ is differentiable at $\theta$ and $\theta+\pi/2$ unless there
is an edge of $\DP$ whose outward unit normal is $u(\theta)$ or $u(\theta+\pi/2)$.
Further, whenever $h_{\DP}$ is differentiable (as a function of $\theta$) it is locally of the
form $(e_1^Tp_i)\cos\theta+(e_2^Tp_i)\sin\theta$, for some $i\in [s]$, and therefore smooth.
Let $\Theta$ denote the set of all $\theta\in [0,\pi)$ for which $\varphi$ is smooth.  For $\theta \in \Theta$, $h''_{\DP}(\theta)=-h_{\DP}(\theta)$. Thus, $\varphi''(\theta)$ is also negative, and therefore $\varphi$ does not have a local minimum for $\theta \in \Theta$.  Hence any minimum of $\varphi$ corresponds to a frame for which at least one direction is parallel to an edge of $P$.
\end{proof}

In higher dimensions, we can argue by induction as in the proof of Theorem \ref{GeneralizedO'Rourke} to obtain the following result.

\begin{thm}\label{GeneralizedO'Rourke1}
If $P$ is a convex polytope in $\R^n$, then a rectangular box containing $P$ with minimal
mean width has at most two opposite facets which contain just one point of $P$.
\end{thm}

A functional which can be associated to $\Phi_1$ in a natural way is $$
    \Lambda_1(K)=\max_{F\in \CF^n}\frac{V_1(K)}{\sum_{i=1}^n V_1(K|u_i^\perp)}.
$$

By \cite{CampiGro}, for every $n$-dimensional convex body $K$,
$$
\Lambda_1(K)\leq 1/(n-1)
$$
with equality if and only if $K$ is a rectangular box, and by \cite{CampiGarGro},
$$
\Lambda_1(K)\geq \Lambda_1(C_n).
$$

The following lemma is analogous to Lemma \ref{LambdaPhi}.

\begin{lem}\label{F_1>G_1}
For every $K\in \CK^n$ we have
$$
\Phi_1(K)\leq (n-1)\,\Lambda_1(K).
$$
In the case of equality there exist a coordinate box $B$ and a rotate $K^g$ of $K$
such that
$$K^g |e_i^\perp = B|e_i^\perp \ , \  i\in [n].$$
\end{lem}

\begin{proof}
Let $g$ be a rotation such that $K^g$ admits a coordinate box $B$ as a minimal mean width rectangular box. Then
$$
\Phi_1(K)= \frac{V_1(K^g)}{V_1(B)} = \frac{(n-1)V_1(K)}{\sum_{i=1}^nV_{1}(B|e_i^\perp)} \leq \frac{(n-1)V_1(K)}{\sum_{i=1}^nV_{1}(K^g|e_i^\perp)} \leq (n-1)\Lambda_1(K).
$$
The assertion for the equality case follows from the above chain of inequalities.
\end{proof}

The upper bound of Lemma \ref{F_1>G_1} should be contrasted with the inequality
$$
   \Phi_1(K)\geq \frac{2\kappa_{n-1}}{\kappa_{n}}
$$
of \cite{Chakerian}, where $\kappa_n=\vol_n(\B^n)$,
which holds with equality if and only if $V_1(B(K;F))$ is independent of $F$.

\vspace{1cm}

{\bf Acknowledgement:} The authors are grateful to the referees for their valuable
comments on a previous version of the present paper.

\vspace{1cm}

\bibliographystyle{plain}
\bibliography{references19}

\end{document}